\documentclass{sig-alternate-05-2015}

\usepackage[utf8]{inputenc}

\usepackage{amsmath,amssymb}
\usepackage{amsthmnoproof}
\usepackage{mathrsfs}
\usepackage{amsrefs}
\usepackage[usenames,dvipsnames]{color}
\usepackage{stmaryrd}
\usepackage{enumerate}
\usepackage[algoruled,vlined,english,linesnumbered]{algorithm2e}
\usepackage[pdfpagelabels,colorlinks=true,citecolor=blue]{hyperref}
\usepackage{comment}
\usepackage{multirow}
\usepackage{tikz}

\newcommand{\noopsort}[1]{}

\DeclareMathOperator{\GL}{GL}
\DeclareMathOperator{\val}{val}

\DeclareMathOperator{\tr}{Tr}
\DeclareMathOperator{\com}{Com}

\DeclareMathOperator{\rec}{rec}
\DeclareMathOperator{\cond}{cond}
\DeclareMathOperator{\disc}{Disc}
\DeclareMathOperator{\row}{row}
\DeclareMathOperator{\col}{col}

\newcommand{\Z}{\mathbb Z}
\newcommand{\Zp}{\Z_p}
\newcommand{\Q}{\mathbb Q}
\newcommand{\Qp}{\Q_p}

\newcommand{\OK}{\mathcal{O}_K}

\newcommand{\llb}{[\mkern-2.5mu[}
\newcommand{\rrb}{]\mkern-2.5mu]}
\newcommand{\llp}{(\mkern-2.5mu(}
\newcommand{\rrp}{)\mkern-2.5mu)}

\newcommand{\calU}{\mathcal{U}}

\newcommand{\softO}{O\tilde{~}}

\newcommand{\inv}{\text{\rm inv}}
\newcommand{\app}{\text{\rm app}}

\def\todo#1{\ \!\!{\color{red} #1}}
\definecolor{purple}{rgb}{0.6,0,0.6}

\CopyrightYear{2017}
\setcopyright{acmcopyright}
\conferenceinfo{ISSAC '17,}{July 25-28, 2017, Kaiserslautern, Germany}
\isbn{***-*-****-****-0/17/07}\acmPrice{\$15.00}
\doi{http://dx.doi.org/10.1145/*******.*******}

\permission{Publication rights licensed to ACM. ACM acknowledges that this contribution was authored or co-authored by an employee, contractor or affiliate of a national government. As such, the Government retains a nonexclusive, royalty-free right to publish or reproduce this article, or to allow others to do so, for Government purposes only.}

\begin{document}

\newtheorem{theo}{Theorem}[section]
\newtheorem{lem}[theo]{Lemma}
\newtheorem{prop}[theo]{Proposition}
\newtheorem{cor}[theo]{Corollary}
\newtheorem{quest}[theo]{Question}
\newtheorem{conj}[theo]{Conjecture}
\theoremstyle{definition}
\newtheorem{rem}[theo]{Remark}
\newtheorem{ex}[theo]{Example}
\newtheorem{deftn}[theo]{Definition}

\title{Characteristic polynomials of p-adic matrices}

\numberofauthors{3}
\author{
\alignauthor Xavier Caruso\\
  \affaddr{Universit\'e Rennes 1}\\
  \affaddr{\textsf{xavier.caruso@normalesup.org}}
\alignauthor David Roe\\
  \affaddr{University of Pittsburg}\\
  \affaddr{\textsf{roed@pitt.edu}}
\alignauthor Tristan Vaccon\\
  \affaddr{Universit\'e de Limoges}\\
  \affaddr{\textsf{tristan.vaccon@unilim.fr}}
}

\maketitle

\begin{abstract}
We analyze the precision of the characteristic polynomial $\chi(A)$ of
an $n \times n$ $p$-adic matrix $A$ using differential precision methods
developed previously.  When $A$ is integral with precision $O(p^N)$,
we give a criterion (checkable in time $\softO(n^\omega)$) for
$\chi(A)$ to have precision exactly $O(p^N)$.  We also give a $\softO(n^3)$
algorithm for determining the optimal precision when the criterion is not
satisfied, and give examples when the precision is larger than $O(p^N)$.
\end{abstract}

\vspace{-1mm}

\vspace{-1.5mm}

%
%

\section{Introduction}

The characteristic polynomial is a fundamental invariant of a matrix: its roots
give the eigenvalues, and the trace and determinant can be extracted from its coefficients.
In fact, the best known division-free algorithm for computing determinants
over arbitrary rings \cite{kaltofen:92a} does so using the characteristic polynomial.
Over $p$-adic fields, computing the characteristic polynomial is a key ingredient in
algorithms for counting points of varieties over finite fields (see
\cite{kedlaya:01a, harvey:07a, harvey:14a}.


When computing with $p$-adic matrices, the lack of infinite memory implies
that the entries may only be approximated at some finite precision $O(p^N)$.  As a
consequence, in designing algorithms for such matrices one must analyze
not only the running time of the algorithm but also the accuracy of the result.

Let $M \in M_n(\Qp)$ be known at precision $O(p^N)$.
The simplest approach for computing the characteristic polynomial of $M$
is to compute $\det(X-M)$ either using recursive row expansion or various division free
algorithms \citelist{\cite{seifullin:02a} \cite{kaltofen:92a}}.  There are two issues
with these methods.  First, they are slower than alternatives that allow division,
requiring $O(n!)$, $O(n^4)$ and $\softO(n^{2 + \omega/2})$ operations.  Second,
while the lack of division implies that the result is accurate modulo $p^N$ as
long as $M \in M_n(\Zp)$, they still do not yield the optimal precision.

A faster approach over a field is to compute the Frobenius normal form of $M$,
which is achievable in running time $\softO(n^\omega)$ \cite{storjohann:01a}.
However, the fact that it uses division frequently leads to catastrophic losses of precision.
In many examples, no precision remains at the end of the calculation.

Instead, we separate the computation of the precision of $\chi_M$ from the computation
of an approximation to $\chi_M$.  Given some precision on $M$, we use \cite{caruso-roe-vaccon:14a}*{Lem. 3.4}
to find the best possible precision for $\chi_M$.  The analysis of this precision is the subject
of much of this paper.  With this precision known, the actual calculation of $\chi_M$
may proceed by lifting $M$ to a temporarily higher precision and then using a sufficiently
stable algorithm (see Remark \ref{rem:lift_for_optimal}).

One benefit of this approach is that we may account for diffuse precision:
precision that is not localized on any single coefficient of $\chi_M$.  For example,
let $0 \le \alpha_1 \le \alpha_2 \le \dots \le \alpha_n$, consider a
diagonal matrix $D$ with diagonal entries $(p^{\alpha_1}, \dots, p^{\alpha_n})$,
let $P, Q \in \GL_n(\Zp)$ and set $M = PDQ$.  The valuation of the coefficient
of $X^{n-k}$ in $\chi_M$ will be $\sum_{i=1}^k \alpha_i$, and if $\alpha_{n-1} > 0$
and $M$ is known with precision $O(p^N)$ then the constant term of $\chi_M$
will be known with precision larger than $O(p^N)$ (see
\cite{caruso-roe-vaccon:15a}*{Prop. 3.2}).

As long as none of the eigenvalues of $M$ are congruent to $-1$ modulo $p$,
then none of the coefficients of the characteristic polynomial of $1+M$ will have
precision larger than $O(p^N)$.  But $\chi_{1+M}(X) = \chi_M(X-1)$, so the precision
content of these two polynomials should be equivalent. The solution is that the
extra precision in $\chi_{1+M}$ is diffuse and not visible on any individual coefficient.
We formalize this phenomenon using lattices; see Section \ref{sec:padicprec} for
further explanation, and \cite{caruso:17a}*{\S 3.2.2} for a specific example of the
relationship between $\chi_M$ and $\chi_{1+M}$.

\subsection*{Previous contributions.}

Since the description of Kedlaya's algorithm in
\cite{kedlaya:01a}, the computation of characteristic polynomials
over $p$-adic numbers
has become a crucial ingredient in many counting-points algorithms.
For example, \cite{harvey:07a, harvey:14a} use $p$-adic cohomology
and the characteristic polynomial of Frobenius to compute zeta functions
of hyperelliptic curves.

In most of these papers, the precision analysis usually
deals with great details on how to obtain the matrices (\textit{e.g.} of action of Frobenius)
that are involved in the point-counting schemes.
However, the computation of their characteristic polynomials is often
a little bit less thoroughly studied: some
refer to fast algorithms (using division), while others
apply division-free algorithms.

In \cite{caruso-roe-vaccon:15a}, the authors have begun the application
of the theory of differential precision of \cite{caruso-roe-vaccon:14a}
to the stable computation of characteristic polynomials.
They have obtained a way to express the optimal precision
on the characteristic polynomial, but have not given practical algorithms
to attain this optimal precision.

\subsection*{The contribution of this paper.}

Thanks to the application the framework of differential precision
of \cite{caruso-roe-vaccon:14a} in \cite{caruso-roe-vaccon:15a},
we know that the precision of the characteristic polynomial
$\chi_M$ of a matrix $M \in M_n(\mathbb{Q}_p)$
is determined by the comatrix $\com(X -M).$
In this article, we provide:
\begin{enumerate}
\renewcommand{\itemsep}{0pt}
\item Proposition \ref{prop:shortcom}: a factorization of $\com(X -M)$
as a product of two rank-$1$ matrices (when $M$ has a cyclic vector),
computable in $\softO(n^\omega)$ operations by Theorem \ref{thm:compute_shortcom}
\item Corollary \ref{cor:prec_gain}: a simple, $\softO(n^\omega)$ criterion
to decide whether $\chi_M$ is defined at precision
higher than the precision of $M$ (when $M \in M_n(\mathbb{Z}_p)$).
\item Theorem \ref{thm:alg3}: a $\softO(n^3)$ algorithm with operations in $\Zp$
to compute the optimal precision on each coefficient of $\chi_M$ (when $M$
is given with uniform precision on its entries).
\item Proposition \ref{prop:opteigenvalues}: a $\softO(n^\omega)$
algorithm to compute the optimal precision on each eigenvalue of $M$
\end{enumerate}

\subsection*{Organization of the article.}

In Section \ref{sec:theo_study}, we review the differential
theory of precision developed in \cite{caruso-roe-vaccon:14a}
and apply it to the specific case of the characteristic polynomial,
giving conditions under which the differential will be surjective
(and thus provide a good measure of precision).  We also give
a condition based on reduction modulo $p$ that determines
whether the characteristic polynomial will have a higher
precision that the input matrix, and show that the image
of the set of integral matrices has the structure of an $\OK[X]$-module
when $M$ is itself integral. Finally, we give a compact description
of $\com(X-M)$, the main ingredient in the differential.

In Section \ref{sec:diffHess}, we develop $\softO(n^3)$ algorithms to approximate
the Hessenberg form of $M$, and through it to find $\com(X-M)$ and thus
find the precision of the characteristic polynomial of $M$.  In Section \ref{sec:diffFrob},
we give a $\softO(n^\omega)$ algorithm to compute the compact description of $\com(X-M)$.

Finally, we propose in Section \ref{sec:optjagged}
algorithms to compute the optimal coefficient-wise precision
for the characteristic polynomial.  We also give the results
of some experiments demonstrating that these methods can lead
to dramatic gains in precision over standard interval arithmetic.
We close with results describing the precision associated to
eigenvalues of a matrix.

\subsection*{Notation} Throughout the paper, $K$ will refer to a complete,
discrete valuation field, $\val : K \twoheadrightarrow \Z \cup \{+\infty\}$ to its valuation,
$\OK$ its ring of integers and $\pi$ a uniformizer. 
We will write that $f(n)=\softO (g(n))$ if there exists some
$k \in \mathbb{N}$ such that $f(n)=O(g(n) \log (n)^k).$
We will write $M$ for an $n \times n$ matrix over $K$,
and $\chi$ the characteristic polynomial map, $\chi_M \in K[X]$ for
the characteristic polynomial of $M$ and $d\chi_M$ for the differential
of $\chi$ at $M$, as a linear map from $M_n(K)$ to the space
of polynomials of degree less than $n$.
We fix an $\omega \in \mathbb{R}$ such that the multiplication
of two matrices over a ring is in $O(n^\omega)$ operations in the
ring. Currently, the smallest known $\omega$ is less than
$2.3728639 $ thanks to \cite{legall:14a}. 
We will denote by $I_n$ the identity matrix of rank $n$ in $M_n(K).$
When there is no ambiguity, we will drop this $I_n$
for scalar matrices, \textit{e.g.} for $\lambda \in K$
and $M \in M_n (K),$ $\lambda-M$ denotes $\lambda I_n - M.$
Finally, we write $\sigma_1(M), \dots, \sigma_n(M)$ for the
elementary divisors of $M$, sorted in increasing order of valuation.

\section{Theoretical study}
\label{sec:theo_study}

\subsection{The theory of p-adic precision} \label{sec:padicprec}

We recall some of the definitions and results of \cite{caruso-roe-vaccon:14a}
as a foundation for our discussion of the precision for the characteristic polynomial
of a matrix.  We will be concerned with two $K$-manifolds in what follows:
the space $M_n(K)$ of $n \times n$ matrices with entries in $K$ and the space
$K_n[X]$ of monic degree $n$ polynomials over $K$.  Given a matrix $M \in M_n(K)$,
the most general kind of precision structure we may attach to $M$ is a
\emph{lattice} $H$ in the tangent space at $M$.  However, representing an
arbitrary lattice requires $n^2$ basis vectors, each with $n^2$ entries.  We therefore
frequently work with certain classes of lattices, either \emph{jagged} lattices
where we specify a precision for each matrix entry or \emph{flat} lattices where
every entry is known to a fixed precision $O(p^N)$.  Similarly, precision for
monic polynomials can be specified by giving a lattice in the tangent space
at $f(X) \in K_n[X]$, or restricted to jagged or flat precision in the interest
of simplicity.

Let $\chi : M_n(K) \to K_n[X]$ be the characteristic polynomial map.
Our analysis of the precision behavior of $\chi$ rests upon
the computation of its derivative $d\chi$, using \cite{caruso-roe-vaccon:14a}*{Lem. 3.4}.
For a matrix $M \in M_n(K)$, we identify the tangent space $V$
at $M$ with $M_n(K)$ itself, and the tangent space $W$ at $\chi_M$ with
the space $K_{<n}[X]$ of polynomials of degree less than $n$.
Let $\com(M)$ denote the comatrix of $M$ (when $M \in \GL_n(K)$,
we have $\com(M) = \det(M) M^{-1}$) and $d\chi_M$ the differential at $M$.  Recall
\citelist{\cite{caruso-roe-vaccon:14a}*{Appendix B} \cite{caruso-roe-vaccon:15a}*{\S 3.3}}
that $d\chi_M$ is given by
\begin{equation} \label{eq:dchi}
d\chi_M: dM \mapsto \tr(\com(X-M) \cdot dM).
\end{equation}

\begin{prop}
\label{prop:surjectivity}
For $M \in M_n(K)$, the following conditions are equivalent:
\begin{enumerate}[(i)]
\renewcommand{\itemsep}{0pt}
\item \label{pt:surj} the differential $d\chi_M$ is surjective
\item \label{pt:cyc} the matrix $M$ has a cyclic vector (\emph{i.e.} $M$ is similar
to a companion matrix)
\item \label{pt:1d} the eigenspaces of $M$ over the algebraic
closure $\bar{K}$ of $K$ all have dimension $1$
\item \label{pt:charmin} the characteristic polynomial of $M$ is equal to the minimal polynomial of $M$.
\end{enumerate}
\end{prop}

\begin{proof}
The equivalence of \eqref{pt:cyc}, \eqref{pt:1d}, and \eqref{pt:charmin} is standard; see
\cite{hoffman-kunze:LinearAlgebra}*{\S 7.1} for example.  We now show
\eqref{pt:cyc} $\Rightarrow$ \eqref{pt:surj} and \eqref{pt:surj} $\Rightarrow$ \eqref{pt:1d}

For any $A \in \GL_n(K)$, the image of $d\chi$ at $M$ will be the same
as the image of $d\chi$ at $AMA^{-1}$, so we may assume that
$M$ is a companion matrix.  For a companion matrix, the bottom row of
$\com(X-M)$ consists of $1, X, X^2, \dots, X^{n-1}$ so $d\chi_M$ is surjective.

Now suppose that $M$ has a repeated eigenvalue $\lambda$ over $\bar{K}$.
After conjugating into Jordan normal form over $\bar{K}$, the entries of $\com(X-M)$
will also be block diagonal, and divisible within each block by the product of $(X-\mu)^{d_\mu}$,
where $\mu, d_\mu$ ranges over the eigenvalues and dimensions of the other Jordan blocks.
Since $\lambda$ occurs in two Jordan blocks,
$X - \lambda$ will divide every entry of $\com(X-M)$ and $d\chi_M$ will not be surjective.
\end{proof}

We also have an analogue of Proposition \ref{prop:surjectivity} for integral matrices.

\begin{prop}
\label{prop:intsurj}
For $M \in M_n(\OK)$, the following conditions are equivalent:
\begin{enumerate}[(i)]
\renewcommand{\itemsep}{0pt}
\item \label{pt:surjp} the image of $M_n(\OK)$ under $d\chi_M$ is $\OK[X] \cap K_{<n}[X]$.
\item \label{pt:cycp} the reduction of $M$ modulo $\pi$ has a cyclic vector.
\end{enumerate}
\end{prop}
\begin{proof}
The condition \eqref{pt:surjp} is equivalent to the surjectivity of $d\chi_M$ modulo $\pi$.  The equivalence with
\eqref{pt:cycp} follows the same argument as Proposition \ref{prop:surjectivity}, but over the residue field of $K$.
\end{proof}

Write $B^-_V(r)$ (resp. $B_V(r)$) for the open (resp. closed) ball of radius $r$ in $V$, and let
$\sigma_1(M), \dots, \sigma_n(M)$ denote the elementary divisors of $M$.

\begin{prop} \label{prop:mainlem}
Suppose that $M \in M_n(K)$ satisfies one of the conditions in Proposition \ref{prop:surjectivity}, and let 
\[
\alpha = \min\left(\prod_{i=1}^{n-1} \left\vert \sigma_i(M) \right\vert, 1\right).
\]
Then, for all $\rho \in (0, 1]$ and all
$r \in (0, \alpha^{-1} \cdot \rho^{-1})$, any lattice $H$ such that $B_V^-(\rho r) \subset H \subset B_V(r)$
satisfies:
\begin{equation}
\chi(M + H) = \chi(M) + d\chi_M(H).
\end{equation}
\end{prop}

\begin{proof}

Recall \cite{caruso-roe-vaccon:15a}*{Def. 3.3} that the \emph{precision polygon}
of $M$ is the lower convex hull of the Newton polygons of the entries of $\com(X-M)$.
By \cite{caruso-roe-vaccon:15a}*{Prop. 3.4}, the endpoints of the precision polygon
occur at height $0$ and $\sum_{i=1}^{n-1} \val(\sigma_i(M))$.  By convexity,
$B_W(1) \subset d\chi_M(B_V(\alpha^{-1}))$.

Since the coefficients of $\chi_M$ are given by polynomials in the entries of $M$
with integral coefficients, \cite{caruso-roe-vaccon:15a}*{Prop. 2.2} implies
the conclusion.
\end{proof}

The relationship between precision and the images of lattices under $d\chi_M$ allows us to
apply Proposition \ref{prop:intsurj} to determine when the precision of the characteristic polynomial
is the minimum possible.

\begin{cor} \label{cor:prec_gain}
Suppose that $M \in \GL_n(\OK)$ is known with precision $O(\pi^m)$.
Then the characteristic polynomial of $M$ has precision lattice strictly contained in $O(\pi^m)$
if and only if the reduction of $M$ modulo $\pi$ does not have a cyclic vector.
\end{cor}

Note that this criterion is checkable using $\softO(n^\omega)$ operations in the residue field \cite{storjohann:01a}.

\subsection{Stability under multiplication by $X$}

By definition, the codomain of $d \chi_M$ is $K_{< n}[X]$. 
However, when $M$ is given, $K_{< n}[X]$ is canonically isomorphic
to $K[X]/\chi_M(X)$ as a $K$-vector space. For our purpose, it will 
often be convenient to view $d \chi_M$ as an $K$-linear mapping
$M_n(K) \to K[X]/\chi_M(X)$.

\begin{prop}
Let $A$ be the subring of $K[X]$ consisting of polynomials $P$ for
which $P(M) \in M_n(\OK)$, and $V = d \chi_M \big(M_n(\OK)\big)$
as a submodule of $K[X]/\chi_M(X)$.  Then $V$ is stable under 
multiplication by $A$.
\end{prop}
\begin{proof}
Let $C = \com(X-M)$ and $P \in A$.  By \eqref{eq:dchi}, $V$ is given by
the $\OK$-span of the entries of $C$.   Using the fact that the product of matrix
with its comatrix is the determinant, $(X - M) \cdot C = \chi_M$ and thus
$P(X) \cdot C \equiv P(M) \cdot C \pmod{\chi_M(X)}$.  The span of the entries
of the left hand side is precisely $P(X) \cdot V$, while the span of the entries
of the right hand side is contained within $V$ since $P(M) \in M_n(\OK)$.
\end{proof}

\begin{cor}
If $M \in M_n(\OK)$, then $d \chi_M \big(M_n(\OK)\big)$ 
is stable under multiplication by $X$ and hence is a module over $\OK[X]$.
\end{cor}

\subsection{Compact form of $d \chi_M$}

Let $\mathscr{C}$ be the companion matrix associated to $\chi_M$:
\begin{equation}
\label{eq:companion}
\mathscr{C} = \left( \begin{matrix}
0 & 1 & 0 & \cdots & 0 \\
0 & 0 & 1 & \ddots & \vdots \\
\vdots & \vdots & \ddots & \ddots & 0 \\
0 & 0 & \cdots & 0 & 1 \\
-a_0 & -a_1 & \cdots & \cdots & -a_{n-1}
\end{matrix} \right)
\end{equation}
with $\chi_M = a_0 + a_1 X + \cdots + a_{n-1} X^{n-1} + X^n$.
By Proposition~\ref{prop:surjectivity}, there exists a matrix 
$P \in \GL_n(K)$ such that $M = P \mathscr{C} P^{-1}$. Applying the same
result to the transpose of $M$, we find that there exists another
invertible matrix $Q \in \GL_n(K)$ such that $M^t = Q \mathscr{C} Q^{-1}$.

\begin{prop}
\label{prop:shortcom}
We keep the previous notations and assumptions.
Let $V$ be the row vector $(1, X, \ldots, X^{n-1})$. Then
\begin{equation}
\label{eq:shortcom}
\com(X{-}M) = \alpha \cdot P V^t \cdot V Q^t
\mod \chi_M
\end{equation}
for some $\alpha \in K[X]$.
\end{prop}

\begin{proof}
Write $C = \com(X{-}M)$. From $(X{-}M) \cdot C \equiv 0 
\pmod{\chi_M}$, we deduce $(X{-}\mathscr{C}) \cdot P^{-1} C \equiv 0 \pmod{\chi_M}$. 
Therefore each column of $P^{-1} C$ lies in the right kernel of $X{-}\mathscr{C}$
modulo $\chi_M$. On the other hand, a direct computation shows that
every column vector $W$ lying in the right kernel of $X{-}\mathscr{C}$ modulo 
$\chi_M$ can be written as $W = w \cdot V^t$ for some $w \in 
K[X]/\chi_M$. We deduce that $C \equiv P \cdot V^t B \pmod{\chi_M}$
for some row vector $B$.
Applying the same reasoning with $M^t$, we find that $B$ can be
written $B = \alpha V Q^t$ for some $\alpha \in K[X]/\chi_M$ and
we are done.
\end{proof}

Proposition~\ref{prop:shortcom} shows that $\com(X{-}M)$ can be encoded 
by the datum of the quadruple $(\alpha, P, Q, \chi_M)$ whose total size 
stays within $O(n^2)$ : the polynomials $\alpha$ and $\chi_M$ are 
determined by $2n$ coefficients while we need $2n^2$ entries to 
write down the matrices $P$ and $Q$. 
We shall see moreover in Section \ref{sec:diffFrob} that interesting
information can be read off of this short form $(\alpha, P, Q, 
\chi_M)$.

\begin{rem}
With the previous notation, if $U \in GL_n(K),$
the quadruple for 
$UMU^{-1}$ is
$(\alpha, UP, (U^t)^{-1}Q, \chi_M),$
which can be computed in $O(n^\omega)$ operations in $K.$
This is faster than computing $U \com(X-M) U^{-1},$
which is, at first sight, in $O(n^4)$ operations in $K.$
\end{rem}

\section{Differential\\via Hessenberg form}
\label{sec:diffHess}

In this section, we combine the computation of a Hessenberg form
of a matrix and the computation of the inverse through the Smith normal form (SNF)
over a complete discrete valuation field (CDVF)
to compute $\com(X-M)$ and $d \chi$.
If $M \in M_n(\OK)$, then only division by invertible
elements of $\OK$ will occur.

\subsection{Hessenberg form}

We begin with the computation of
an approximate Hessenberg form.

\begin{deftn}
A \emph{Hessenberg matrix} is a matrix $M \in M_n(K)$ with
\[
M_{i,j}=0 \mbox{ for $j \le i-2$.}
\]
Given integers $n_{i,j}$, an \emph{approximate Hessenberg matrix}
is a matrix $M \in M_n(K)$ with
\[
M_{i,j} = O(\pi^{n_{i,j}}) \mbox{ for $j \le i-2$.}
\]
If $M \in M_n (K)$ and $H \in M_n (K)$ is an (approximate) Hessenberg matrix
similar to $M$, we say that H is an \emph{(approximate) Hessenberg form} of $M.$
\end{deftn}

It is not hard to prove that every matrix over a field admits
a Hessenberg form.
We prove here that over $K,$ if a matrix is 
known at finite (jagged) precision,
we can compute an approximate Hessenberg form of it.
Moreover, we can provide an exact change of basis matrix.
It relies on the following algorithm.

\noindent\hrulefill

\noindent {\bf Algorithm 1:} {\tt Approximate Hessenberg form computation}

\noindent{\bf Input:} a matrix $M$ in $M_n(K).$

\smallskip

\noindent 0.\ $P:=I_n.$ \: $H:=M.$

\noindent 1.\ {\bf for} $j=1,\dots,n-1$ {\bf do} 

\noindent 2.\  \:  {\bf swap} the row $j+1$ with a row $i_{min}$ ($i_{min} \geq 2$) s.t. $\val(H_{i_{min},j})$ is minimal. 

\noindent 3.\  \:  {\bf for} $i=j+2,\dots,n$ {\bf do} 

\noindent 4. \ \: \:  \textbf{Eliminate} the significant digits of $H_{i,j}$ by pivoting with row $j+1$ 
using a matrix $T.$

\noindent 5. \ \: \:  $H:=H \times T^{-1}.$ \: $P:=T \times P.$

\noindent 6. \textbf{Return} $H,P.$

\vspace{-1ex}\noindent\hrulefill

\medskip

\begin{prop} 
Algorithm 1 computes $H$ and $P$ realizing an approximate Hessenberg form of $M.$
$P$ is exact over finite extensions of $\mathbb{Q}_p$ and $k\llp X\rrp$, and the computation is in $O(n^3)$ operations in $K$ at precision the maximum precision of a coefficent in $M.$
\end{prop}
\begin{proof}
Let us assume that $K$ is a finite extensions of $\mathbb{Q}_p$ or $k\llp X \rrp.$
Inside the nested \textbf{for} loop, if we want to eliminate $\pi^{u_y} \varepsilon_y+O(\pi^{n_y})$ with pivot $\pi^{u_x} \varepsilon_x+O(\pi^{n_x}),$
with the $\varepsilon$'s being units,
the corresponding coefficient of the corresponding shear matrix is the lift(in $\mathbb{Z}, $  $\mathbb{F}_q[X],$ $\mathbb{Q}[X]$ or adequate extension) of $\pi^{u_y-u_x} \varepsilon_y \varepsilon_x^{-1} \mod \pi^{u_y-u_x\min (n_x-u_x,n_y-u_y)}.$
Exactness follows directly. Over other fields, we can not lift, but the computations are still valid.
The rest is clear.
\end{proof}

\begin{rem} \label{rem:char_pol_from_hessenberg}
From a Hessenberg form of $M,$ it is well known
that one can compute the characteristic polynomial of 
$M$ in $O(n^3)$ operations in $K$ \cite{Cohen:2013}*{pp. 55--56}.
However, this computation involves division, and its
precision behavior is not easy to quantify.
\end{rem}

\subsection{Computation of the inverse}

In this section, we prove that to compute the inverse of
a matrix over a CDVF $K$, the Smith normal form is precision-wise optimal in the flat-precision case.
We first recall the differential of matrix inversion.

\begin{lem}
Let $u \: : \: GL_n (K) \rightarrow GL_n(K),$ $M \mapsto M^{-1}.$
Then for $M \in GL_n (K),$ $du_M(dM)=M^{-1} dM M^{-1}.$
It is always surjective.
\end{lem}

We then have the following result about the loss in precision when computing the inverse.

\begin{prop}
Let $\cond(M) = \val(\sigma_n(M))$.
If $dM$ is a flat precision of $O(\pi^m)$ on $M$ then $M^{-1}$
can be computed at precision $O(\pi^{m-2\cond(M)})$ by a \textbf{SNF} computation
and this lower-bound is optimal,
at least when $m$ is large.
\end{prop}
\begin{proof}
The smallest valuation of a coefficient of $M^{-1}$ is $-\cond(M).$
It is $-2\cond(M)$ for $M^{-2}$ and it is then clear that $m-2\cond(M)$
can be obtained as the valuation of a coefficient of $du_M(dM)$
and the smallest that can be achieved this way for $dM$ in a precision lattice
of flat precision. Hence the optimality of the bound given, at least when 
$m$ is large \cite{caruso-roe-vaccon:14a}*{Lem. 3.4}.

Now, the computation of the Smith normal form was described in \cite{Vaccon-these}.
From $M$ known at flat precision $O(\pi^m),$ we can obtain an exact $\Delta$, and $P$ and $Q$ 
known at precision at least $O(\pi^{m-\cond(M)})$, with coefficients in $\OK$
and determinant in $\OK^\times$ realizing an Smith normal form of $M.$
There is no loss in precision when computing $P^{-1}$ and $Q^{-1}.$
Since the smallest valuation occurring in $\Delta^{-1}$ is $-\cond(M),$
we see that $M^{-1}=Q^{-1} \Delta^{-1} P^{-1}$ is known at precision at least $O(\pi^{m-2\cond(M)}),$
which concludes the proof.
\end{proof}

\subsection{The comatrix of $X{-}H$}

In this section, we compute $\com(X-H)$ for a Hessenberg matrix $H$
using the Smith normal form computation of the previous section.
The entries of $\com(X-H)$ lie in $K[X]$, which is not a CDVF,
so we may not directly apply the methods of the previous section.
However, we may relate $\com(X-H)$ to $\com(1-XH)$, whose
entries lie in the CDVF $K\llp X\rrp$.  In this way, we compute $\com(X-H)$
using an SNF method, with no division in $K$.

First, we need a lemma relating comatrices of
similar matrices:

\begin{lem} \label{lem:comatrix_of_similar}
If $M_1,M_2 \in M_n(K)$ and $P \in GL_n (K)$ are such that
$M_1=PM_2P^{-1},$ then:
\[ \com (X-M_1)=P \com (X-M_2) P^{-1}. \] 
\end{lem}

The second ingredient we need is reciprocal polynomials.
We extend its definition to matrices of polynomials.
\begin{deftn}
Let $d \in \mathbb{N}$ and $P \in K[X]$ of degree at most $d.$ 
We define the reciprocal polynomial of order $d$ of $P$ as $P^{\rec,d}=X^d P \left( 1/X \right).$
Let $A \in M_n(K[X])$ a matrix of polynomials of degree at most $d.$
We denote by $A^{\rec,d}$ the matrix with $(A^{\rec,d})_{i,j} = (A_{i,j})^{\rec,d}$.
\end{deftn}
We then have the following result :
\begin{lem}
Let $M \in M_n(K).$ Then:
\begin{eqnarray*}
\com(1-XM)^{\rec,n-1}=&\com(X-M), \\
(\chi_M I_n)^{\rec,n}=&(1-XM) \com(1-XM).\\
\end{eqnarray*}
\end{lem}
\begin{proof}
It all comes down to the following result:
let $A \in M_d(K[X])$ a matrix of polynomials of degree at most $1,$
then $\det (A^{\rec,1})=\det(A)^{\rec,d}.$
Indeed, one can use multilinearity of the determinant on $X^d \det(A(1/X))$
to prove this result.
It directly implies the second part of the lemma; the first part follows
from the fact that the entries of $\com(X-M)$ and of $\com(1-XM)$
are determinants of size $n-1$.
\end{proof}

This lemma allows us to compute $\com(1-XM)$ instead of $\com(X-M).$
This has a remarkable advantage: the pivots during the computation of
the SNF of $\com(1-XM)$ are units of $\OK\llb X\rrb,$ and are known
in advance to be on the diagonal. This leads to a very smooth
precision and complexity behaviour when the input matrix lives in 
$M_n(\OK).$ 

\noindent\hrulefill

\noindent {\bf Algorithm 2:} {\tt Approximate $\com (X -H)$ }

\noindent{\bf Input:} an approximate Hessenberg matrix $H$ in $M_n(\OK).$

\smallskip

\noindent 0.\ $U:=1-XH.$ $U_0:=1-XH.$

\noindent 1.\ While updating $U$, \textbf{track} $P$ and $Q$ so that $U_0=PUQ$ is always satisfied.

\noindent 2.\ {\bf for} $i=1,\dots,n-1$ {\bf do} 

\noindent 3.\  \:  \textbf{Eliminate}, modulo $X^{n+1}$ the coefficients $U_{i,j},$ for $j\geq i+1$ 
using the invertible pivot
$U_{i,i}=1+XL_{i,i} \mod X^{n+1}$ (with $L_{i,i} \in \OK[X]$). 

\noindent 4.\    {\bf for} $i=1,\dots,n-1$ {\bf do} 

\noindent 5. \  \:  \textbf{Eliminate}, modulo $X^{n+1}$ the coefficients $U_{i+1,i},$
using the invertible pivot $U_{i,i}.$

\noindent 6. \ $\psi:=\prod_i U_{i,i}.$

\noindent 7. \ Rescale to get $U = I_n \mod X^{n+1}.$

\noindent 8. \ $V:=\psi \times P \times Q   \mod X^{n+1}.$ \footnote{The product $P \times Q$
should be implemented by sequential row operations corresponding to the eliminations in Step 5
in order to avoid a product of two matrices in $M_n(\OK[X])$.}

\noindent 9. \ \textbf{Return} $V^{\rec,n-1}, \psi^{\rec,n}.$

\vspace{-1ex}\noindent\hrulefill

\medskip

\begin{theo}
Let $H \in M_n(\OK)$ be an approximate Hessenberg matrix.
Then, using Algorithm 2, one can compute $\com (X -H)$
and $\chi_H$ in 
$\softO (n^3)$ operations in $\OK$ at the precision given by $H.$
\end{theo}
\begin{proof}
First, the operations of the lines 2 and 3 use $\softO (n^3)$ operations
in $\OK$ at the precision given by $H.$
Indeed, since $H$ is an approximate Heisenberg matrix, when we use $U_{i,i}$ as pivot
the only other nonzero coefficient in its column is $U_{i+1,i}$.
As a consequence, when performing this column-pivoting, only two rows ($i$ and
$i+1$) lead to operations in $\OK \llb X \rrb$ other than checking precision.
Hence, line 3 costs $\softO (n^2)$ for the computation of $U.$
Following line 1, the computation of $Q$ is done by operations on rows, starting from the identity matrix.
The order in which the entries of $U$ are cleared implies that $Q$ is just filled in as an upper triangular matrix:
no additional operations in $\OK\llb X\rrb$ are required. Thus the total cost
for lines 2 and 3 is indeed $\softO (n^3)$ operations.

For lines 4 and 5, there are only $n-1$ eliminations, resulting in a $\softO (n^2)$ cost
for the computation of $U.$. Rather than actually construct $P$, we just track the eliminations
performed in order to do the corresponding row operations on $Q$, since we only need the product $P \times Q$.

Line 6 is in $\softO (n^2)$ and 7 in $\softO (n^3).$

Thanks to the fact that the $P$ only corresponds to the product of $n-1$ 
shear matrices, the
product on line 8 is in $\softO (n^3).$
We emphasize that no division has been done throughout the algorithm.
Line 9 is costless, and the result is then proved.
\end{proof}
\begin{rem}
If $M \in M_n(K)$ does not have coefficients in $\OK,$
we may apply Algorithms 1 and 2 to $p^v M \in M_n(\OK)$
in $\softO (n^3)$ operations in $\OK$, and then divide
the coefficient of $X^k$ in the resulting polynomial by $p^{kv}$.
\end{rem}

We will see in Section \ref{sec:optjagged} that for an
entry matrix with coefficients known at flat precision,
Algorithms 1 and 2 are enough to
know the optimal jagged precision on $\chi_M.$

\subsection{The comatrix of $X{-}M$}

In this section, we combine Proposition \ref{prop:shortcom}
with Algorithm 2 to compute the comatrix of $X-M$
when $\chi_M$ is squarefree.   Note that this condition on
$\chi_M$ is equivalent to $M$ being diagonalizable under
the assumption that $d\chi_M$ is surjective.  The result is
the following $\softO(n^3)$ algorithm, where the only
divisions are for gcd and modular inverse computations.

\noindent\hrulefill

\noindent {\bf Algorithm 3:} {\tt Approximate $\com (X{-}M)$ }

\noindent{\bf Input:} an approx. $M \in M_n(\OK),$ with $\disc(\chi_M) \neq 0.$ 

\smallskip

\noindent 0.\ Find $P \in GL_n(\OK)$ and $H \in M_n(\OK),$ approximate Hessenberg,
such that $M=PHP^{-1},$ using Algorithm 1. 

\noindent 1.\ Compute $A=\com (X-H)$ and $\chi_M = \chi_H$ using Algorithm 2.

\noindent 2.\ Do $\row(A,1) \leftarrow \row(A,1)+\sum_{i=2}^n \mu_i \row(A,i),$ for
random $\mu_i \in \OK,$ by doing $T \times A$ for some $T \in GL_n(\OK).$
Compute $B:=TAT^{-1}.$

\noindent 3.\ Similarily compute $C:=S^{-1}BS$ for $S \in GL_n(\OK)$ corresponding to
adding a random linear combination of the columns of index $j \ge 2$
to the first column of $B.$ 

\noindent 4.\  \textbf{If} $\gcd(C_{1,1}, \chi_M) \neq 1,$ \textbf{then} go to 2.

\noindent 5. Let $F$ be the inverse of $C_{1,1} \mod \chi_M$.

\noindent 6. Let $U := \col(C,1)$ and $V := F \cdot \row(C,1) \mod \chi_M$.

\noindent 7. Return $\com(X-M):=(PT^{-1}S U \times V S^{-1} T P^{-1}) \mod \chi_M.$

\vspace{-1ex}\noindent\hrulefill

\medskip

\begin{theo}
\label{thm:alg3}
For $M \in M_n(\OK)$ such that $\disc( \chi_M) \neq 0,$
Algorithm 3 computes
$\com (X-M) \pmod{\chi_M}$ in average complexity $\softO (n^3)$
operations in $K$.
The only divisions occur in taking gcds and inverses modulo $\chi_M$.
\end{theo}
\begin{proof}
As we have already seen, completing Steps 0 and 1 is in $\softO (n^3).$
Multiplying by $T$ or $S$ or their inverse corresponds
to $n$ operations on rows or columns over a matrix with coefficients
in $\OK[X]$ of degree at most $n.$
Thus, it is in $\softO (n^3).$
Step 5 is in $\softO (n),$ Step 6 in $\softO(n^2)$ and Step 7 in $\softO(n^3)$.
All that is to prove is that the set of $P$ and $S$ to avoid
is of dimension at most $n-1.$
The idea is to work modulo $X-\lambda$
for $\lambda$ a root of $\chi (M)$ (in an algebraic closure)
and then apply Chinese Remainder Theorem.
The goal of the Step $2$ is to ensure the first row of $B$ contains an
invertible entry modulo $\chi_M.$
Since $A(\lambda)$ is of rank one, the $\mu_i$'s have to avoid an
affine hyperplane so that $\row(B,1) \mod (X-\lambda)$ is a non-zero vector.
Hence for $\row(B,1) \mod \chi (M)$ to contain an invertible coefficient,
a finite union of affine hyperplane is to avoid.
Similarly, the goal of Step 3 is to put an invertible coefficient (modulo
$\chi_M$) on $C_{1,1},$ and again, only a finite union of affine
hyperplane is to avoid.
Hence, the set that the $\mu_i$'s have to avoid is a finite union
of hyperplane, and hence, is of dimension at most $n-1.$
Thus, almost any choice of $\mu_i$ leads to a matrix $C$ passing the test
in Step 4.
This concludes the proof.
\end{proof}

\begin{rem}
As in the previous section, it is possible
to scale $M \in M_n(K)$ so as to
get coefficients in $\OK$ and apply the previous algorithm.
\end{rem}
\begin{rem}
We refer to \cite{caruso:15a} for the handling
of the precision of gcd and modular inverse computations.
In this article, ways to tame the loss of precision
coming from divisions are explored, following
the methods of \cite{caruso-roe-vaccon:14a}.
\end{rem}

\section{Differential\\via Frobenius form}
\label{sec:diffFrob}

The algorithm designed in the previous section computes the differential 
$d \chi_M$ of $\chi$ at a given matrix $M \in M_n(K)$ for a cost of 
$O(n^3)$ operations in $K$. This seems to be optimal given that the 
(naive) size of the $d \chi_M$ is $n^3$: it is a matrix of size $n
\times n^2$. It turns out however that improvements are still possible!
Indeed, thanks to Proposition~\ref{prop:shortcom}, the matrix of 
$d \chi_M$ admits a compact form which can be encoded using only $O(n^2)$ 
coefficients. The aim of this short section is to design a fast 
algorithm (with complexity $\softO(n^\omega)$) for computing this short 
form. The price to pay is that divisions in $K$ appear, which can be an 
issue regarding to precision in particular cases.
In this section, we only estimate the number of operations in $K$
and not their behaviour on precision.

From now on, we fix a matrix $M \in M_n(K)$ for which $d \chi_M$ is 
surjective. Let $(\alpha, P, Q, \chi_M)$ be the quadruple encoding
the short form of $d \chi_M$; we recall that they are related by the
relations:
\begin{align*}
d \chi_M(dM) & =\tr(\com(X{-}M) \cdot dM) \\
\com(X{-}M) & = \alpha \cdot P V^t \cdot V Q^t \mod \chi_M.
\end{align*}
An approximation to $\chi_M$ can be computed in $\softO(n^\omega)$ operations in $K$
(\textit{e.g.} as a by-product of \cite{storjohann:01a}).

The matrix $P$ can be computed as follows. Pick $c \in K^n$. Define 
$c_i = M^i c$ for all $i \geq 1$. The $c_i$'s can be computed in 
$\softO(n^\omega)$ operations in $K,$ \textit{e.g.}
using the first algorithm
of \cite{keller-gehrig:85a}. Let $P_\inv$ be the 
$n \times n$ matrix whose rows are the $c_i$'s for $1 \leq i \leq n$. 
Remark that $P_\inv$ is invertible if and only if $(c_0, c_1, \ldots, 
c_{n-1})$ is a basis of $K^n$ if and only if $c$ is a cyclic vector. 
Moreover after base change to the basis $(c_0, \ldots, c_{n-1})$, the matrix 
$M$ takes the shape \eqref{eq:companion}. In other words, if $P_\inv$
is invertible, then $P = P_\inv^{-1}$ is a solution of $M = P \mathscr{C} P^{-1}$,
where $\mathscr{C}$ is the companion matrix similar to $M$.
Moreover, observe that the condition ``$P_\inv$ is invertible'' is open
for the Zariski topology. It then happens with high probability as soon
as it is not empty, that is as soon as $M$ admits a cyclic vector, which
holds by assumption.

The characteristic polynomial $\chi_M$ can be recovered thanks to the
relation $a_0c_0 + a_1c_1 + \dots + a_{n-1}c_{n-1} = -c_{n-1} \cdot P$.

Now, instead of directly computing $Q$, we first compute a matrix $R$ 
with the property that $\mathscr{C}^t = R \mathscr{C} R^{-1}$. To do so,
we apply the same strategy as above except that we start with the vector
$e = (1, 0, \ldots, 0)$ (and not with a random vector). A simple computation shows
that, for $1 \leq i \leq n{-}1$, the vector $\mathscr{C}^i e$ has the shape:
$$\mathscr{C}^i e = (0, \ldots, 0, -a_0, \star, \ldots, \star)$$
with $n{-}i$ starting zeros. Therefore the $\mathscr{C}^i e$'s form a basis of
$K^n$, \emph{i.e.} $e$ is always a cyclic vector of $\mathscr{C}$. Once $R$ has
been computed, we recover $Q$ using the relation $Q = P_\inv^t R$.

It remains to compute the scaling factor $\alpha$. For this, we write
the relation:
\begin{equation}
\label{eq:shortcomN}
\com(X{-}\mathscr{C}) = \alpha \cdot V^t \cdot V R^t \mod \chi_M
\end{equation}
which comes from Eq.~\eqref{eq:shortcom} after multiplication on the 
left by $P^{-1}$ and multiplication on the right by $P$. We observe
moreover that the first row of $R$ is $(1, 0, \ldots, 0)$. Evaluating
the top left entry of Eq.~\eqref{eq:shortcomN}, we end up with the 
relation:
$$\alpha = a_1 + a_2 X + \cdots + a_{n-1} X^{n-2} + X^{n-1}.$$
No further computation are then needed to derive the value of $\alpha$.
We summarize this section with the following theorem:

\begin{theo}
\label{thm:compute_shortcom}
Given $M \in M_n(K)$ such that $d \chi_M$ is surjective,
then one can compute $(\alpha, P, Q, \chi_M)$
in $K[X]$ such that
$\com(X{-}N) = \alpha \cdot V^t \cdot V R^t \mod \chi_M$
in $\softO(n^\omega)$ operations in $K.$
\end{theo}

\section{Optimal jagged precision}
\label{sec:optjagged}

In the previous Sections, \ref{sec:diffHess} and \ref{sec:diffFrob},
we have proposed algorithms to obtain the comatrix of 
$X-M.$ Our motivation for these computations is to then
be able to understand what is the optimal precision on $\chi_M.$
In this section, we provide some answers to this question,
along with numerical evidence.
We also show that it is then possible to derive 
optimal precision of eigenvalues of $M.$ 

\subsection{On the characteristic polynomial}

For $0 \leq k < n$, let $\pi_k : K[X] \to K$ be the mapping taking a 
polynomial to its coefficients in $X^k$. By applying 
\cite{caruso-roe-vaccon:14a}*{Lem. 3.4} to the composite $\pi_k 
\circ \chi_M$, one can figure out the optimal precision on the
$k$-th coefficient of the characteristic polynomial of $M$ (at
least if $M$ is given at enough precision).

Let us consider more precisely the case where $M$ is given at 
jagged precision: the $(i,j)$ entry of $M$ is given at precision 
$O(\pi^{N_{i,j}})$ for some integers $N_{i,j}$. 
Lemma 3.4 of \cite{caruso-roe-vaccon:14a} then shows that
the optimal precision on the $k$-th coefficient of $\chi_M$ is 
$O(\pi^{N'_k})$ where $N'_k$ is given by the formula:
\begin{equation}
\label{eq:optjagged}
N'_k = \min_{1 \leq i, j\leq n} N_{j,i} + \val(\pi_k(C_{i,j})),
\end{equation}
where $C_{i,j}$ is the $(i,j)$ entry of the comatrix $\com(X{-}M)$.

\begin{prop} \label{prop:optimal_jagged}
If $M \in M_n(\OK)$ is given at (high enough) jagged precision, 
then we can compute the optimal jagged precision on $\chi_M$ in 
$\softO (n^3)$ operations in $K$.
\end{prop}

\begin{proof}
We have seen in \S \ref{sec:diffHess} and \S \ref{sec:diffFrob}
that the computation of the matrix $C = \com(X{-}M)$ can be carried out 
within $\softO(n^3)$ operations in $K$ (either with the Hessenberg 
method or the Frobenius method). We conclude by applying 
Eq.~\eqref{eq:optjagged} which requires no further operation in $K$
(but $n^3$ evaluations of valuations and $n^3$ manipulations of 
integers).
\end{proof}

\begin{rem}
If $M \in M_n(\OK)$ is given at (high enough) \emph{flat} precision, 
then we can avoid the final base change step in the Hessenberg method.
Indeed, observe that, thanks to Lemma \ref{lem:comatrix_of_similar}, 
we can write:
$$\tr(\com (X{-}M) \cdot dM)=\tr(\com (X{-}H)\cdot P^{-1} dM P)$$
where $P$ lies in $\GL_n(\OK)$. Moreover, the latter condition implies
that $P^{-1} dM P$ runs over $M_n(\OK)$ when $P$ runs over $M_n(\OK)$.
As a consequence, the integer $N'_k$ giving the optimal precision on the 
$k$-th coefficient of $M$ is also equal to
$N + \min_{1 \leq i, j\leq n} \val(\pi_k(C^H_{i,j}))$
where $C^H_{i,j}$ is the $(i,j)$ entry of $\com(X{-}H)$,
where $H$ is the Hessenberg form of $M$.
\end{rem}

\begin{rem} \label{rem:lift_for_optimal}
As a consequence of the previous discussion, once the optimal jagged 
precision is known, it is possible to lift the entries of $M$ to a sufficiently
large precision, rescale them to have entries in $O_K$
and then use Algorithm 2 to
compute the characteristic polynomial.
The output might then need to be rescaled and
 truncated at the optimal precision. 
This requires $\softO(n^3)$ operations in $O_K$ 
and unfortunately, for several instances, may require to increase a lot the precision.
\end{rem}

\medskip

\noindent
{\bf Numerical experiments.}
We have made numerical experiments in \textsc{SageMath}~\cite{sage}
in order to compare the optimal precision obtained with the methods
explained above with the actual precision obtained by the software.
For doing so, we picked a sample of $1000$ random matrices $M$ in 
$M_9(\Q_2)$ where all the entries are given at the same \emph{relative
precision}.
We recall that, in \textsc{SageMath}, random elements $x \in \Q_p$ are 
generated as follows. We fix an integer $N$ --- the so-called relative
precision --- and generate elements of $\Q_p$ of the shape
$$x = p^v \cdot \big(a + O\big(p^{N+v_p(a)}\big)\big)$$
where $v$ is a random integer generated according to the distribution:
$$\mathbb P [v = 0] = \frac 1 5 \quad ; \quad
\mathbb P [v = n] = \frac 2 {5\cdot |n| \cdot (|n|+1)} \text{ for }
|n| \geq 1$$
and $a$ is an integer in the range $[0, p^N)$, selected uniformly at random.

Once this sample has been generated, we computed, for each $k \in \{0, 
1, \ldots, 8\}$, the three following quantities:

\vspace{-2mm}

\begin{enumerate}[$\bullet$]
\renewcommand{\itemsep}{0pt}
\item the optimal precision on the $k$-th coefficient of the 
characteristic polynomial of $M$ given by Eq.~\eqref{eq:optjagged}
\item in the capped relative model\footnote{Each 
coefficient carries its own precision which is updated after each 
elementary arithmetical operation.},
the precision gotten on the $k$-th coefficient of the 
characteristic polynomial of $M$ computed \emph{via} the call:

\hfill$M\texttt{.charpoly(algorithm="df")}$\hfill\null

\item in the model of floating-point arithmetic (see 
\cite{caruso:17a}*{\S 2.3}), the number of correct digits of the 
$k$-th coefficient of the characteristic polynomial of $M$.
\end{enumerate}

\begin{rem}
The keyword \texttt{algorithm="df"} forces \texttt{SageMath} to use the 
division free algorithm of \cite{seifullin:02a}. It is likely that, proceeding so, we 
limit the loss of precision.
\end{rem}

\noindent
The table of Figure~\ref{fig:exp} summarizes the results obtained. 
\begin{figure}
\hfill
{\renewcommand*{\arraystretch}{1.3}
\begin{tabular}{|c|r|r|r|}
\cline{2-4} 
\multicolumn{1}{l|}{} 
  & \multicolumn{3}{c|}{Average loss of accuracy} \\
\cline{2-4}
\multicolumn{1}{l|}{\null\hspace{6mm}\null} 
  & \makebox[1.5cm]{\hfill Optimal\hfill\null}
  & \makebox[1.5cm]{\hfill CR\hfill\null}
  & \makebox[1.5cm]{\hfill FP\hfill\null} \\
\hline
\rule{0pt}{2.7ex}%
$X^0$ & $3.17$ & $196\phantom{.00}$ & $189\phantom{.00}$ \vspace{-1.5ex} \\
{\scriptsize (det.)} 
& {\scriptsize dev: $1.76$} & {\scriptsize dev: $240$} & {\scriptsize dev: $226$} \\
\rule{0pt}{2.7ex}%
$X^1$ & $2.98$ & $161\phantom{.00}$ & $156\phantom{.00}$ \vspace{-1.5ex} \\
& {\scriptsize dev: $1.69$} & {\scriptsize dev: $204$} & {\scriptsize dev: $195$} \\
\rule{0pt}{2.7ex}%
$X^2$ & $2.75$ & $129\phantom{.00}$ & $126\phantom{.00}$ \vspace{-1.5ex} \\
& {\scriptsize dev: $1.57$} & {\scriptsize dev: $164$} & {\scriptsize dev: $164$} \\
\rule{0pt}{2.7ex}%
$X^3$ & $2.74$ & $108\phantom{.00}$ & $105\phantom{.00}$ \vspace{-1.5ex} \\
& {\scriptsize dev: $1.73$} & {\scriptsize dev: $144$} & {\scriptsize dev: $143$} \\
\rule{0pt}{2.7ex}%
$X^4$ & $2.57$ &  $63.2\phantom{0}$ &  $60.6\phantom{0}$ \vspace{-1.5ex} \\
& {\scriptsize dev: $1.70$} & {\scriptsize dev: $85.9$} & {\scriptsize dev: $85.8$} \\
\rule{0pt}{2.7ex}%
$X^5$ & $2.29$ &  $51.6\phantom{0}$ &  $49.7\phantom{0}$ \vspace{-1.5ex} \\
& {\scriptsize dev: $1.66$} & {\scriptsize dev: $75.3$} & {\scriptsize dev: $74.9$} \\
\rule{0pt}{2.7ex}%
$X^6$ & $2.07$ &   $9.04$           &   $8.59$ \vspace{-1.5ex} \\
& {\scriptsize dev: $1.70$} & {\scriptsize dev: $26.9$} & {\scriptsize dev: $26.4$} \\
\rule{0pt}{2.7ex}%
$X^7$ & $1.64$ &   $5.70$           &   $5.38$ \vspace{-1.5ex} \\
& {\scriptsize dev: $1.65$} & {\scriptsize dev: $15.3$} & {\scriptsize dev: $14.7$} \\
\rule{0pt}{2.7ex}%
$X^8$ & $0.99$ &   $0.99$           &   $0.99$ \vspace{-1.5ex} \\
{\scriptsize (trace)} 
& {\scriptsize dev: $1.37$} & {\scriptsize dev: $1.37$} & {\scriptsize dev: $1.37$} \\
\hline
\end{tabular}}
\hfill\null

\medskip

\hfill
{\footnotesize Results for a sample of $1000$ instances}
\hfill\null

\caption{Average loss of accuracy on the coefficients
of the characteristic polynomial of a random $9 \times 9$
matrix over $\Q_2$}
\label{fig:exp}

\end{figure}
It should be read as follows. First, the acronyms CR and FP 
refers to ``capped relative'' and ``floating-point'' respectively.
The numbers displayed in the table are the average loss of
\emph{relative} precision. 
More precisely, if $N$ is the relative precision at
which the entries of the input random matrix $M$ have been generated
and $v$ is the valuation of the $k$-th coefficient of $\chi_M$, then:

\vspace{-2mm}

\begin{enumerate}[$\bullet$]
\renewcommand{\itemsep}{0pt}
\item the column ``Optimal'' is the average of the quantities 
$(N'_k{-}v) - N$ (where $N'_k$ is defined by Eq.~\eqref{eq:optjagged}): 
$N'_k{-}v$ is the optimal \emph{relative} precision, so that the
difference $(N'_k{-}v) - N$ is the loss of relative precision;
\item the column ``CR'' is the average of the quatities 
$(\text{CR}_k{-}v) - N$ where $\text{CR}_k$ is the computed (absolute) 
precision on the $k$-th coefficient of $\chi_M$;
\item the column ``FP'' is the average of the quatities 
$(\text{FP}_k{-}v) - N$ where $\text{FP}_k$ is the first position of
an incorrect digit on the $k$-th coefficient of $\chi_M$.
\end{enumerate}

We observe that the loss of relative accuracy stays under control in the 
``Optimal'' column whereas it has a very erratic behavior --- very large 
values and very large deviation as well --- in the two other columns. 
These experiments thus demonstrate the utility of the methods developed
in this paper.

\subsection{On eigenvalues}
\label{ssec:eigenvalues}

Let $M \in M_n(K)$ and $\lambda \in K$ be a \emph{simple}
\footnote{the corresponding generalized eigenspace has dimension $1$} eigenvalue 
of $M$. We are interesting in quantifying the optimal precision on 
$\lambda$ when $M$ is given with some uncertainty.

To do so, we fix an approximation $M_\app \in M_n(K)$ of $M$ and 
suppose that the uncertainty of $M$ is ``jagged'' in the sense that
each entry of $M$ is given at some precision $O(\pi^{N_{i,j}})$.
Let $\lambda_\app$ be the relevant eigenvalue of $M_\app$. We remark 
that it is possible to follow the eigenvalue $\lambda_\app$ on a small 
neighborhood $\calU$ of $M$. More precisely, there exists a unique continuous 
function $f : \calU \to K$ such that:
\begin{enumerate}[$\bullet$]
\renewcommand{\itemsep}{0pt}
\item $f(M_\app) = \lambda_\app$, and
\item $f(M')$ is an eigenvalue of $M'$ for all $M' \in \calU$.
\end{enumerate}

\begin{lem}
The function $f$ is strictly differentiable on a neighborhood of 
$M_\app$.
The differential of $f$ at $M$ is the linear mapping:
$$dM \mapsto d \lambda = - \frac{\tr(\com(\lambda-M) \cdot dM)}
{\chi'_M(\lambda)}$$
where $\chi'_M$ is the usual derivative of $\chi_M$.
\end{lem}

\begin{proof}
The first assertion follows from the implicit function Theorem.
Differentiating the relation $\chi_M(\lambda) = 0$, we get
$\chi'_M(\lambda) \cdot d \lambda + \tr(\com(X-M) \cdot dM)(\lambda) = 0$,
from which the Lemma follows.
\end{proof}

\noindent
Lemma 3.4 of \cite{caruso-roe-vaccon:14a} now implies that, if the
$N_{i,j}$'s are large enough and sufficiently well balanced, the optimal
precision on the eigenvalue $\lambda$ is $O(\pi^{N'})$ with:
$$N' = \min_{1 \leq i, j\leq n} \big(N_{j,i} + \val(C_{i,j}(\lambda)) - 
\val(\chi'_M(\lambda))\big)$$
where $C_{i,j}$ denotes as above the $(i,j)$ entry of $\com(X{-}M)$.
Writing
$\com(X{-}M) = \alpha \cdot P V^t \cdot V Q^t \text{ mod } \chi_M$
as in Proposition~\ref{prop:shortcom}, we find:
\begin{align}
N' & = \val(\alpha(\lambda)) - \val(\chi'_M(\lambda)) \nonumber \\
& \hspace{2mm} + \min_{1 \leq i, j\leq n} \big(N_{j,i} + 
\val(P_i V(\lambda)^t) + \val(V(\lambda) Q_j^t)\big) \label{eq:Nprime}
\end{align}
where $P_i$ denotes the $i$-th row of $P$ and, similarly, $Q_j$
denotes the $j$-th row of $Q$. Note moreover that $V(\lambda)$ is
the row vector $(1, \lambda, \ldots, \lambda^{n-1})$.
By the discussion of \S \ref{sec:diffFrob}, the exact value of $N'$ can be 
determined for a cost of $\softO(n^\omega)$ operations in $K$ and
$O(n^2)$ operations on integers. 

When $M$ is given at flat precision, \emph{i.e.} the $N_{i,j}$'s are all 
equal to some $N$, the formula for $N'$ may be rewritten:
\begin{align}
N' & = N + \val(\alpha(\lambda)) - \val(\chi'_M(\lambda)) \nonumber \\
& \hspace{2mm} + \min_{1 \leq i \leq n} \val(P_i V(\lambda)^t)
+ \min_{1 \leq j \leq n} \val(V(\lambda) Q_j^t)\big) \label{eq:Nprimeflat}
\end{align}
and can therefore now be evaluated for a cost of $\softO(n^\omega)$
operations in $K$ and only $O(n)$ operations with integers.

\medskip

To conclude, let us briefly discuss the situation where we want
to figure out the optimal jagged precision on a tuple $(\lambda_1,
\ldots, \lambda_s)$ of simple eigenvalues. Applying \eqref{eq:Nprime},
we find that the optimal precision on $\lambda_k$ is 
\begin{align*}
N'_k & = \val(\alpha(\lambda_k)) - \val(\chi'_M(\lambda_k)) \\
& \hspace{2mm} + \min_{1 \leq i, j\leq n} \big(N_{j,i} + 
\val(P_i V(\lambda_k)^t) + \val(V(\lambda_k) Q_j^t)\big).
\end{align*}

\begin{prop}
\label{prop:opteigenvalues}
The $N'_k$'s can be all computed in $\softO(n^\omega)$ operations
in $K$ and $O(n^2 s)$ operations with integers.

If the $N_{i,j}$'s are all equal, the above complexity can be 
lowered to $\softO(n^\omega)$ operations in $K$ and $O(n s)$ operations 
with integers.
\end{prop}

\begin{proof}
The $\alpha(\lambda_k)$'s and the $\chi'_M(\lambda_k)$'s can be computed 
for a cost of $\softO(ns)$ operations in $K$ using fast multipoint 
evaluation methods (see 10.7 of \cite{gathen-gerhard:13a}).
On the other hand, we observe that $P_i V(\lambda_k)^t$ is nothing but
the $(i,k)$ entry of the matrix:
$$P \cdot \left( \begin{matrix}
\lambda_1 & \cdots & \lambda_s \\
\lambda_1^2 & \cdots & \lambda_s^2 \\
\vdots & & \vdots \\
\lambda_1^{n-1} & \cdots & \lambda_s^{n-1}
\end{matrix} \right).$$
The latter product can be computed in $\softO(n^\omega)$ operations in 
$K$\footnote{It turns out that $\softO(n^2)$ is also possible because
the right factor is a structured matrix (a truncated Vandermonde):
computing the above product reduces to evaluating a polynomial at the
points $\lambda_1, \ldots, \lambda_s$.}. Therefore all the $P_i 
V(\lambda_k)^t$'s (for $i$ and $k$ varying) 
can be determined with the same complexity. 
Similarly all the $V(\lambda) Q_j^t$ are computed for the same cost.
The first assertion of Proposition~\ref{prop:opteigenvalues} follows.
The second assertion is now proved similarly to the case of a unique
eigenvalue.
\end{proof}

\bibliographystyle{plain}
\bibliography{charpoly}

\end{document}